\documentclass[11pt,letterpaper]{article}
\usepackage[english]{babel}
\usepackage{amsmath}
\usepackage{amsfonts}
\usepackage{amssymb}
\usepackage{verbatim}
\usepackage{graphicx}
\usepackage{color}

\numberwithin{equation}{section}
\newtheorem{theorem}{Theorem}[section]

\newtheorem{proposition}[theorem]{Proposition}

\newtheorem{definition}{Definition}
\newtheorem{remark}[theorem]{Remark}
\newenvironment{proof}[1][Proof]{\textbf{#1.} }{\ \rule{0.5em}{0.5em}}

\newcommand{\var}{{\rm Var} \mspace{1mu}}

\begin{document}
\title{Efficient simulation for branching linear recursions}

\author{Ningyuan Chen and Mariana Olvera-Cravioto \\ [12pt]
Industrial Engineering and Operations Research\\
Columbia University \\
New York, NY 10027, USA\\
}

\maketitle

\begin{abstract}
We consider a linear recursion of the form
$$R^{(k+1)}\stackrel{\mathcal D}{=}\sum_{i=1}^{N}C_iR^{(k)}_i+Q,$$
where $(Q,N,C_1,C_2,\dots)$ is a real-valued random vector with $N\in\mathbb{N}=\{0, 1, 2, \dots\}$, $\{R^{(k)}_i\}_{i\in\mathbb{N}}$ is a sequence of i.i.d. copies of $R^{(k)}$, independent of $(Q,N,C_1,C_2,\dots)$, and $\stackrel{\mathcal{D}}{=}$ denotes equality in distribution. For suitable vectors $(Q,N,C_1,C_2,\dots)$ and provided the initial distribution of $R^{(0)}$ is well-behaved, the process $R^{(k)}$ is known to converge to the endogenous solution of the corresponding stochastic fixed-point equation, which appears in the analysis of information ranking algorithms, e.g., PageRank, and in the complexity analysis of divide and conquer algorithms, e.g. Quicksort. Naive Monte Carlo simulation of $R^{(k)}$ based on the branching recursion has exponential complexity in $k$, and therefore the need for efficient methods. We propose in this paper an iterative bootstrap algorithm that has linear complexity and can be used to approximately sample $R^{(k)}$. We show the consistency of estimators based on our proposed algorithm.
\end{abstract}

\section{Introduction}
\label{sec:intro}

The complexity analysis of divide and conquer algorithms such as Quicksort \cite{Rosler_91,Fill_Jan_01,Ros_Rus_01} and the more recent analysis of information ranking algorithms on complex graphs (e.g., Google's PageRank) \cite{Volk_Litv_08,Jel_Olv_10,Chen_Litv_Olv_14} motivate the analysis of the stochastic fixed-point equation
\begin{equation}\label{eq:SFPE}
    R\stackrel{\mathcal D}{=} \sum_{r=1}^{N} C_r R_r + Q,
\end{equation}
where $(Q, N, C_1, C_2, \dots)$ is a real-valued random vector with $N \in \mathbb{N}$, and $\{R_i\}_{i \in \mathbb{N}}$ is a sequence of i.i.d. copies of $R$, independent of $(Q, N, C_1, C_2, \dots)$. More precisely, the number of comparisons required in Quicksort for sorting an array of length $n$, properly normalized, satisfies in the limit as the array's length grows to infinity a distributional equation of the form in \eqref{eq:SFPE}. In the context of ranking algorithms, it has been shown that the rank of a randomly chosen node in a large directed graph with $n$ nodes converges in distribution, as the size of the graph grows, to $R$, where $N$ represents the in-degree of the chosen node and the $\{C_i\}_{i \geq 1}$ are functions of the out-degrees of its neighbors.

Although equation \eqref{eq:SFPE} is known to have multiple solutions, and an extensive amount of literature has been devoted to their characterization (see e.g. \cite{Als_Big_Mei_10,Als_Big_Mei_10,Alsm_Mein_10b} and the references therein), in applications we are often interested only in the so-called endogenous solution. This solution can be shown to be the unique limit under iterations of the distributional recursion
\begin{equation} \label{eq:Iterations}
R^{(k+1)} \stackrel{\mathcal D}{=} \sum_{i=1}^{N} C_i R^{(k)}_i + Q,
\end{equation}
where $(Q,N,C_1,C_2,\dots)$ is a real-valued random vector with $N\in\mathbb{N}$, and $\{R^{(k)}_i\}_{i\in\mathbb{N}}$ is a sequence of i.i.d. copies of $R^{(k)}$, independent of $(Q,N,C_1,C_2,\dots)$, provided one starts with an initial distribution for $R^{(0)}$ with sufficient finite moments (see, e.g., Lemma 4.5 in \cite{Jel_Olv_12b}).  Moreover, asymptotics for the tail distribution of the endogenous solution $R$ are available under several different sets of assumptions for $(Q, N, C_1, C_2, \dots)$ \cite{Jel_Olv_10,Jel_Olv_12a,Jel_Olv_12b,Olvera_12}. However, no approximations exist for the distribution of $R$  besides its tail behavior, and even the calculation of its non-integer/absolute moments can be difficult. Hence the need to design efficient numerical methods to compute relevant statistics.

As will be discussed later, the endogenous solution to \eqref{eq:SFPE} can be explicitly constructed on a weighted branching process. Thus, drawing some similarities with the analysis of branching processes, and the Galton-Watson process in particular, one could think of using the Laplace transform of $R$ to obtain its distribution. Unfortunately, the presence of the weights $\{C_i\}$ in the Laplace transform
\[
 \varphi(s) =   E\left[ \exp\left( -s R \right) \right] = E\left[\exp\left( -sQ \right) \prod_{i=1}^N \varphi( sC_i)     \right]
\]
 makes its inversion problematic, making a simulation approach even more important.

 The first observation we make regarding the simulation of $R$, is that when $P(Q = 0) < 1$ it is enough to be able to approximate $R^{(k)}$ for fixed values of $k$, since both $R^{(k)}$ and $R$ can be constructed in the same probability space in such a way that the difference $|R^{(k)} - R|$ is geometrically small. More precisely, under very general conditions (see Proposition \ref{P.MomentsR} in Section \ref{sec:wbp}), there exist positive constants $K < \infty$ and $c < 1$ such that
 \begin{equation} \label{eq:geometricTail}
E\left[ \left| R^{(k)} - R \right|^\beta \right] \leq K c^{k+1}.
\end{equation}
Our goal is then to simulate $R^{(k)}$ for a suitably large value of $k$.

The simulation of $R^{(k)}$ is not that straightforward either, since naive Monte Carlo using \eqref{eq:Iterations} starting from some initial distribution $R^{(0)}$ implies the computation of a geometric number of copies of $(Q, N, C_1, C_2, \dots)$, of order $(E[N])^k$ when $E[N] > 1$, which is usually the case in the applications we are interested in. Hence, the naive simulation approach  can be prohibitive. Instead, we propose in this paper an iterative bootstrap algorithm that outputs a sample pool of observations $\{ \hat R^{(k,m)}_i \}_{i=1}^m$ whose empirical distribution converges, in the Kantorovich-Rubinstein distance, to that of $R^{(k)}$ as the size of the pool $m \to \infty$. This mode of convergence is equivalent to weak convergence and convergence of the first absolute moments (see, e.g., \cite{Villani_2009}). Moreover, the complexity of our proposed algorithm is linear in $k$.

The paper is organized as follows. Section~\ref{sec:wbp} describes the weighted branching process and the linear recursion. The algorithm itself is given in Section~\ref{sec:algorithm} .
Section~\ref{sec:convergence} introduces the Kantorovich-Rubinstein distance and proves the convergence properties of our proposed algorithm. Numerical examples to illustrate the precision of the algorithm are presented in Section~\ref{sec:numerical}.

\section{Linear recursions on weighted branching processes} \label{sec:wbp}

As mentioned in the introduction, the endogenous solution to \eqref{eq:SFPE} can be explicitly constructed on a weighted branching process. To describe the structure of a weighted branching process, let  $\mathbb{N}_+ = \{1, 2, 3, \dots\}$ be the set of positive integers and let $U = \bigcup_{k=0}^\infty (\mathbb{N}_+)^k$ be the set of all finite sequences ${\bf i} = (i_1, i_2, \dots, i_n)$, $n\ge 0$, where by convention $\mathbb{N}_+^0 = \{ \emptyset\}$ contains the null sequence $\emptyset$. To ease the exposition, we will use $({\bf i}, j) = (i_1,\dots, i_n, j)$ to denote the index concatenation operation.

Next, let $(Q, N, C_1, C_2, \dots)$ be a real-valued vector with $N \in \mathbb{N}$. We will refer to this vector as the generic branching vector. Now let $\{ (Q_{\bf i}, N_{\bf i}, C_{({\bf i}, 1)}, C_{({\bf i}, 2)}, \dots ) \}_{{\bf i} \in U}$ be a sequence of i.i.d. copies of the generic branching vector. To construct a weighted branching process we start by defining a tree as follows: let $A_0 = \{ \emptyset\}$ denote the root of the tree, and define the $n$th generation according to the recursion
$$A_n = \{ ({\bf i}, i_n) \in U:  {\bf i} \in A_{n-1}, 1 \leq i_n \leq N_{\bf i} \}, \quad n \geq 1.$$
Now, assign to each node ${\bf i}$ in the tree a weight $\Pi_{\bf i}$ according to the recursion
$$\Pi_\emptyset \equiv 1, \qquad \Pi_{({\bf i}, i_n)} = C_{({\bf i}, i_n)} \Pi_{\bf i}, \qquad n \geq 1, $$
see Figure \ref{F.Tree}. If $P(N< \infty)=1$ and $C_i \equiv 1$ for all $i \geq 1$, the weighted branching process reduces to a Galton-Watson process.

\begin{figure}[t]
\centering
\begin{picture}(350,110)(0,0)
\put(0,0){\includegraphics[scale = 0.75, bb = 30 560 510 695, clip]{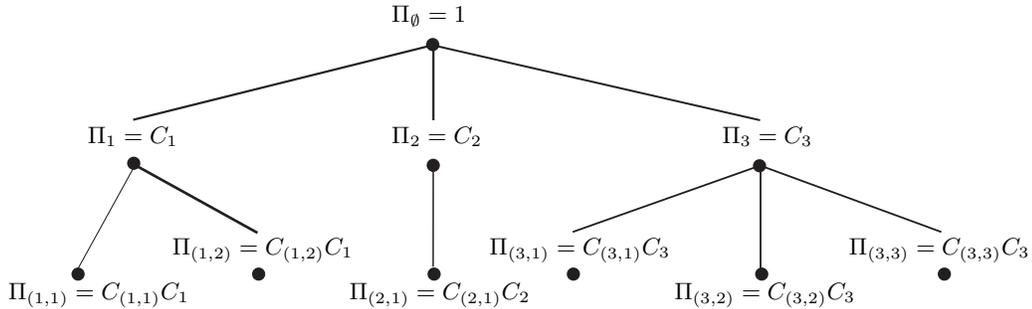}}
\put(145,105){\footnotesize $\Pi_\emptyset = 1$}
\put(30,59){\footnotesize $\Pi_{1} = C_1$}
\put(145,59){\footnotesize $\Pi_{2} = C_2$}
\put(270,59){\footnotesize $\Pi_{3} = C_3$}
\put(0,0){\footnotesize $\Pi_{(1,1)} = C_{(1,1)} C_1$}
\put(62,17){\footnotesize $\Pi_{(1,2)} = C_{(1,2)} C_1$}
\put(129,0){\footnotesize $\Pi_{(2,1)} = C_{(2,1)} C_2$}
\put(182,17){\footnotesize $\Pi_{(3,1)} = C_{(3,1)} C_3$}
\put(252,0){\footnotesize $\Pi_{(3,2)} = C_{(3,2)} C_3$}
\put(318,17){\footnotesize $\Pi_{(3,3)} = C_{(3,3)} C_3$}
\end{picture}
\caption{Weighted branching process}\label{F.Tree}
\end{figure}

For a weighted branching process with generic branching vector $(Q, N, C_1, C_2, \dots)$, define the process $\{ R^{(k)} : k \geq 0\}$ as follows:
\begin{equation} \label{eq:RkDef}
    R^{(k)} = \sum_{j=0}^k \sum_{{\bf i} \in A_j} Q_{\bf i} \Pi_{\bf i}, \quad k \geq 0.
\end{equation}
By focusing on the branching vector belonging to the root node, i.e., $(Q_\emptyset, N_\emptyset, C_1, C_2, \dots)$ we can see that the process
$\{ R^{(k)} \}$ satisfies the distributional equations
\begin{align}
    R^{(0)}&=Q_{\emptyset}\stackrel{\mathcal D}{=}Q\nonumber\\
    R^{(k)} &= \sum_{r=1}^{N_\emptyset} C_r \left( \sum_{j=1}^k \sum_{(r,{\bf i}) \in A_j} Q_{(r,{\bf i})} \Pi_{(r,{\bf i})}/ C_r  \right) + Q_\emptyset \stackrel{\mathcal{D}}{=} \sum_{r=1}^{N} C_r R^{(k-1)}_r + Q, \qquad k \geq 1,\label{eq:dist-equation}
\end{align}
where $R^{(k-1)}_r$ are i.i.d. copies of $R^{(k-1)}$, all independent of $(Q, N, C_1, C_2, \dots)$.  Here and throughout the paper the convention is that $XY/Y \equiv 1$ if $Y = 0$. Moreover, if we define
\begin{equation} \label{eq:EndogenousSol}
R = \sum_{j=0}^\infty \sum_{{\bf i} \in A_j} Q_{\bf i} \Pi_{\bf i},
\end{equation}
we have the following result. We use $x \vee y$ to denote the maximum of $x$ and $y$.

\begin{proposition} \label{P.MomentsR}
Let $\beta \geq 1$ be such that $E[|Q|^\beta] < \infty$ and $E\left[ \left( \sum_{i=1}^N |C_i| \right)^\beta \right] < \infty$. In addition, assume either (i) $(\rho_1 \vee \rho_\beta) < 1$ , or (ii) $\beta = 2$, $\rho_1 = 1$, $\rho_\beta  < 1$ and $E[Q] = 0$. Then, there exist constants $K_\beta > 0$ and $0 < c_\beta < 1$ such that for $R^{(k)}$ and $R$ defined according to \eqref{eq:RkDef} and \eqref{eq:EndogenousSol}, respectively, we have
$$\sup_{k \geq 0} E\left[ |R^{(k)}|^\beta \right] \leq K_\beta < \infty \qquad \text{and} \qquad E\left[ | R^{(k)} - R|^\beta \right] \leq K_\beta c_\beta^{k+1}.$$
\end{proposition}

\begin{proof}
For the case $\rho_1 \vee \rho_\beta < 1$, Lemma 4.4 in \cite{Jel_Olv_12b} gives that for $W_n = \sum_{{\bf i} \in A_n} Q_{\bf i} \Pi_{\bf i}$ and some finite constant $H_\beta$ we have
$$E\left[ \left| W_n \right|^\beta \right] \leq H_\beta (\rho_1 \vee \rho_\beta)^n.$$
Let $c_\beta = \rho_1 \vee \rho_\beta$. Minkowski's inequality then gives
$$\left|\left| R^{(k)} \right|\right|_\beta \leq \sum_{n=0}^k \left|\left| W_n \right|\right|_\beta \leq \sum_{n=0}^\infty \left( H_\beta c_\beta^n \right)^{1/\beta} = \left( \frac{H_\beta}{1- c_\beta^{1/\beta}} \right)^{1/\beta} \triangleq \left( K_\beta \right)^{1/\beta} < \infty.$$
Similarly,
$$\left|\left| R^{(k)} - R \right|\right|_\beta \leq \sum_{n=k+1}^\infty \left|\left| W_n \right|\right|_\beta \leq \sum_{n=k+1}^\infty \left( H_\beta c_\beta^n \right)^{1/\beta} = c_\beta^{(k+1)/\beta} \left( \frac{H_\beta}{1- (\rho_1 \vee \rho_\beta)^{1/\beta}} \right)^{1/\beta} = \left( K_\beta c_\beta^{k+1} \right)^{1/\beta}.$$

For the case $\beta = 2$, $\rho_1 = 1$, $\rho_\beta < 1$ and $E[Q] = 0$ we have that
\begin{align*}
E\left[ W_n^2 \right] &= E\left[ \left( \sum_{r=1}^{N_\emptyset} C_r W_{n-1,r} \right)^2 \right]  = E\left[ \sum_{r=1}^{N_\emptyset} C_r^2 (W_{n-1,r})^2 + \sum_{1\leq r \neq s \leq N_\emptyset} C_r C_s W_{n-1,r} W_{n-1,s} \right] ,
\end{align*}
where $W_{n-1,r} = \sum_{(r,{\bf i}) \in A_n} Q_{(r,{\bf i})} \Pi_{(r,{\bf i})} / C_r$, and the $\{W_{n-1,r}\}_{r \geq 1}$ are i.i.d. copies of $W_{n-1}$, independent of $(N_\emptyset, C_1, C_2, \dots)$. Since $E[ W_{n} ] = 0$ for all $n \geq 0$, it follows that
$$E[W_n^2] = \rho_2 E[W_{n-1}^2] = \rho_2^n E[W_0^2] = \var(Q) \rho_2^n.$$
The two results now follow from the same arguments used above with $H_2 = \var(Q)$ and $c_2 = \rho_2$.
\end{proof}

It follows from the previous result that under the conditions of Proposition  \ref{P.MomentsR}, $R^{(k)}$ converges to $R$ both almost surely and in $L^\beta$-norm. Similarly, if we ignore the $Q$ in the generic branching vector, assume that $C_i \geq 0$ for all $i$, and define the process
$$W^{(k)} = \sum_{{\bf i} \in A_k} \Pi_{\bf i} = \sum_{r=1}^{N_\emptyset} C_r \left( \sum_{(r,{\bf i}) \in A_k} \Pi_{(r,{\bf i})}/ C_r \right) \stackrel{\mathcal{D}}{=} \sum_{r=1}^N C_r W_r^{(k-1)},$$
where the $\{W_r^{(k-1)}\}_{r \geq 1}$ are i.i.d. copies of $W^{(k-1)}$ independent of $(N, C_1, C_2, \dots)$, then it can be shown that $\{ W^{(k)}/\rho_1^k : k \geq 0\}$ defines a nonnegative martingale which converges almost surely to the endogenous solution of the stochastic fixed-point equation
$$W \stackrel{\mathcal D}{=}  \sum_{i=1}^N \frac{C_i}{\rho_1} \, W_i,$$
where the $\{W_i\}_{i \geq 1}$ are i.i.d. copies of $W$, independent of $(N, C_1, C_2, \dots)$. We refer to this equation as the homogeneous case.

As mentioned in the introduction, our objective is to generate a sample of $R^{(k)}$ for values of $k$ sufficiently large to suitably approximate $R$. Our proposed algorithm can also be used to simulate $W^{(k)}$, but due to space limitations we will omit the details.

\section{The algorithm}\label{sec:algorithm}

Note that based on \eqref{eq:RkDef}, one way to simulate $R^{(k)}$ would be to simulate a weighted branching process starting from the root and up to the $k$ generation and then add all the weights $Q_{\bf i} \Pi_{\bf i}$ for ${\bf i} \in \bigcup_{j=0}^k A_j$. Alternatively, we could generate a large enough pool of i.i.d. copies of $Q$ which would represent the $Q_{\bf i}$ for ${\bf i} \in A_k$, and use them to generate a pool of i.i.d. observations of $R^{(1)}$ by setting
$$R^{(1)}_i = \sum_{r=1}^{N_i} C_{(i,r)} R^{(0)}_r + Q_i,$$
where $\{ (Q_i, N_i, C_{(i,1)}, C_{(i,2)}, \dots)\}_{i \geq 1}$ are i.i.d. copies of the generic branching vector, independent of everything else, and the $R^{(0)}_r$ are the $Q$'s generated in the previous step. We can continue this process until we get to the root node. On average, we would need $(E[N])^k$ i.i.d. copies of $Q$ for the first pool of observations, $(E[N])^{k-1}$ copies of the generic branching vector for the second pool, and in general, $(E[N])^{k-j}$ for the $j$th step. This approach is equivalent to simulating the weighted branching process starting from the $k$th generation and going up to the root, and is the result of iterating \eqref{eq:Iterations}.

Our proposed algorithm is based on this ``leaves to root" approach, but to avoid the need for a geometric number of ``leaves", we will resample from the initial pool to obtain a pool of the same size of observations of $R^{(1)}$. In general, for the $j$th generation we will sample from the pool obtained in the previous step of (approximate) observations of $R^{(j-1)}$ to obtain conditionally independent (approximate) copies of $R^{(j)}$. In other words, to obtain a pool of approximate copies of $R^{(j)}$ we bootstrap from the pool previously obtained of approximate copies of $R^{(j-1)}$. The approximation lies in the fact that we are not sampling from $R^{(j-1)}$ itself, but from a finite sample of conditionally independent observations that are only approximately distributed as $R^{(j-1)}$. The algorithm is described below.

Let $(Q,N,C_1,C_2,\dots)$ denote the generic branching vector defining the weighted branching process. Let $k$ be the depth of the recursion that we want to simulate, i.e., the algorithm will produce a sample of random variables approximately distributed as $R^{(k)}$. Choose $m\in \mathbb N_+$ to be the bootstrap sample size.  For each $0 \leq j \leq k$, the algorithm outputs $\mathcal P^{(j,m)} \triangleq \left(\hat R^{(j,m)}_1, \hat R^{(j,m)}_2,\dots,\hat R^{(j,m)}_m\right)$, which we refer to as the sample pool at level $j$.

\begin{enumerate}
 \item[1.] \emph{Initialize}: Set $j=0$. Simulate a sequence $\{Q_i\}_{i = 1}^m$ of i.i.d. copies of $Q$ and let $\hat R^{(0,m)}_i=Q_i$ for $i=1,\dots,m$. Output $\mathcal P^{(0,m)}=\left(\hat R^{(0,m)}_1, \hat R^{(0,m)}_2,\dots,\hat R^{(0,m)}_m\right)$ and update $j = 1$.

\item[2.] While $j \leq k$:

\begin{enumerate}
\item[i)] Simulate a sequence $\{ (Q_i, N_i, C_{(i,1)}, C_{(i,2)}, \dots) \}_{i =1}^m$ of i.i.d. copies of the generic branching vector, independent of everything else.

\item[ii)] Let
        \begin{equation}
            \hat R^{(j,m)}_i = \sum_{r=1}^{N_i} C_{(i,r)} \hat R^{(j-1,m)}_{(i,r)} + Q_i, \qquad i=1,\dots,m,\label{eq:bootstrap-recur}
        \end{equation}
        where the $\hat R^{(j-1,m)}_{(i,r)}$ are sampled uniformly with replacement from the pool $\mathcal P^{(j-1,m)}$.

\item[iii)] Output $\mathcal P^{(j,m)}=\left(\hat R^{(j,m)}_1, \hat R^{(j,m)}_2,\dots,\hat R^{(j,m)}_m\right)$ and update $j = j+1$.

\end{enumerate}

\end{enumerate}

    Bootstrapping refers broadly to any method that relies on random sampling with replacement \cite{bootstrap1993efron}. For example, bootstrapping can be used to estimate the variance of an estimator, by constructing samples of the estimator from a number of resamples of the original dataset with replacement.
    With the same idea, our algorithm draws samples uniformly with replacement from the previous bootstrap sample pool.
    Therefore, the $\hat R^{(j-1,m)}_{(i,r)}$ on the right-hand side of \eqref{eq:bootstrap-recur} are only conditionally independent given $\mathcal P^{(j-1,m)}$. Hence, the samples in $\mathcal P^{(j,m)}$ are identically distributed but not independent for $j\ge 1$.

    As we mentioned earlier, the distribution of the $\{ \hat R_i^{(j,m)}\}$ in $\mathcal{P}^{(j,m)}$ are only approximately distributed as $R^{(j)}$, with the exception of the $\{\hat R_i^{(0,m)}\}$ which are exact. The first thing that we need to prove is that the distribution of the observations in $\mathcal{P}^{(j,m)}$ does indeed converge to that of $R^{(j)}$. Intuitively, this should be the case since the empirical distribution of the $\{\hat R_i^{(0,m)}\}$ is the empirical distribution of $m$ i.i.d. observations of $R^{(0)}$, and therefore should be close to the true distribution of $R^{(0)}$ for suitably large $m$. Similarly, since the $\{ \hat R_i^{(1,m)}\}$ are constructed by sampling from the empirical distribution of $\mathcal{P}^{(0,m)}$, which is close to the true distribution of $R^{(0)}$, then their empirical distribution should be close to the empirical distribution of $R^{(1)}$, which in turn should be close to the true distribution of $R^{(1)}$. Inductively, provided the approximation is good in step $j-1$, we can expect the empirical distribution of $\mathcal{P}^{(j,m)}$ to be close to the true distribution of $R^{(j)}$. In the following section we make the mode of the convergence precise by considering the Kantorovich-Rubinstein distance between the empirical distribution of $\mathcal{P}^{(j,m)}$ and the true distribution of $R^{(j)}$.

    The second technical aspect of our proposed algorithm is the lack of independence among the observations in $\mathcal{P}^{(k,m)}$, since a natural estimator for quantities of the form $E[ h(R^{(k)}) ]$ would be to use
    \begin{equation} \label{eq:hEstimator}
\frac{1}{m} \sum_{i=1}^m h (\hat R_i^{(k,m)}).
\end{equation}
    Hence, we also provide a result establishing the consistency of estimators of the form in \eqref{eq:hEstimator} for a suitable family of functions $h$.

    We conclude this section by pointing out that the complexity of the algorithm described above is of order $k m$, while the naive Monte Carlo approach has order $(E[N])^k m$. This is a huge gain in efficiency.

%

\section{Convergence and consistency}
\label{sec:convergence}

In order to show that our proposed algorithm does indeed produce observations that are approximately distributed as $R^{(k)}$ for any fixed $k$, we will show that the empirical distribution function of the observations in $\mathcal{P}^{(k,m)}$ , i.e.,
\[
    \hat F_{k,m}(x) = \frac{1}{m} \sum_{i=1}^{m} 1( \hat R_{i}^{(k,m)}\le x )
\]
converges as $m \to \infty$ to the true distribution function of $R^{(k)}$, which we will denote by $F_k$. We will show this by using the Kantorovich-Rubinstein distance, which is a metric on the space of probability measures. In particular, convergence in this sense is equivalent to weak convergence plus convergence of the first absolute moments.


\begin{definition} \label{d.wasserstein}
    let $M(\mu, \nu)$ denote the set of joint probability measures on $\mathbb R \times \mathbb R$  with marginals $\mu$ and $\nu$. then, the Kantorovich-Rubinstein distance between $\mu$ and $\nu$ is given by
$$d_1(\mu, \nu) = \inf_{\pi \in M(\mu, \nu)} \int_{\mathbb R\times \mathbb R} | x - y | \, d \pi(x, y).$$
\end{definition}

We point out that $d_1$ is only strictly speaking a distance when both $\mu$ and $\nu$ have finite first absolute moments. Moreover, it is well known that
  \begin{equation} \label{eq:Explicit}
  d_1(\mu, \nu) = \int_{0}^1 | F^{-1}(u) - G^{-1}(u) | du = \int_{-\infty}^\infty | F(x) - G(x) | dx,
  \end{equation}
where $F$ and $G$ are the cumulative distribution functions of $\mu$ and $\nu$, respectively, and $f^{-1}(t) = \inf\{ x \in \mathbb{R}: f(x) \geq t\}$ denotes the pseudo-inverse of $f$. It follows that the optimal coupling of two real random variables $X$ and $Y$ is given by $(X, Y) = (F^{-1}(U), G^{-1}(U))$, where $U$ is uniformly distributed in $[0, 1]$.

\begin{remark}
The Kantorovich-Rubinstein distance is also known as the Wasserstein metric of order 1. In general, both the Kantorovich-Rubinstein distance and the more general Wasserstein metric of order $p$ can be defined in any metric space; we restrict our definition in this paper to the real line since that is all we need. We refer the interested reader to \cite{Villani_2009} for more details.
\end{remark}

With some abuse of notation, for two distribution functions $F$ and $G$ we use $d_1(F,G)$ to denote the Kantorovich-Rubinstein distance between their corresponding probability measures.

The following proposition shows that for i.i.d. samples, the expected value of the Kantorovich-Rubinstein distance between the empirical distribution function and the true distribution converges to zero.

\begin{proposition}\label{P.empirical_d1}
  Let $\{ X_i\}_{i \geq 1}$ be a sequence of i.i.d. random variables with common distribution $F$. Let $F_n$ denote the empirical distribution function of a sample of size $n$. Then, provided there exists $\alpha\in (1,2)$ such that $E\left[ |X_1|^{\alpha} \right]< \infty$, we have that
$$E\left[ d_1(F_n,F) \right]\le n^{-1 + 1/\alpha}\left( \frac{2\alpha}{\alpha-1}+\frac{2}{2-\alpha} \right) E[|X_1|^\alpha].$$
\end{proposition}

Proposition~\ref{P.empirical_d1} can be proved following the same arguments used in the proof of Theorem 2.2 in \cite{del1999central} by setting $M=1$, and thus we omit it.

We now give the main theorem of the paper, which establishes the convergence of the expected Kantorovich-Rubinstein distance between $\hat F_{k,m}$ and $F_k$. Its proof is based on induction and the explicit representation \eqref{eq:Explicit}. Recall that $\rho_\beta = E\left[ \sum_{i=1}^N |C_i|^\beta \right]$.

\begin{theorem}\label{T.d1-l1-convergence}
    Suppose that the conditions of Proposition \ref{P.MomentsR} are satisfied for some $\beta > 1$. Then, for any $\alpha \in (1, 2)$ with $\alpha \leq \beta$, there exists a constant $K_\alpha < \infty$ such that
\begin{equation} \label{eq:InductionHyp}
E\left[ d_1(\hat F_{k,m}, F_{k}) \right] \leq K_\alpha m^{-1+1/\alpha} \sum_{i=0}^k \rho_1^i.
\end{equation}
\end{theorem}

\begin{proof}
By Proposition \ref{P.MomentsR} there exists a constant $H_\alpha$ such that
$$H_\alpha = \sup_{k \geq 0} E\left[ |R^{(k)}|^\alpha \right] \leq \sup_{k \geq 0} \left( E\left[ |R^{(k)}|^\beta \right] \right)^{\alpha/\beta} < \infty.$$
Set $K_\alpha = H_\alpha \left( \frac{2\alpha}{\alpha-1} + \frac{2}{2-\alpha} \right)$. We will give a proof by induction.

For $j = 0$, we have that
$$\hat F_{0,m}(x) = \frac{1}{m} \sum_{i=1}^m 1(Q_i \leq x),$$
where $\{Q_i\}_{i \geq 1}$ is a sequence of i.i.d. copies of $Q$. It follows that $\hat F_{0,m}$ is the empirical distribution function of $R^{(0)}$, and by Proposition \ref{P.empirical_d1} we have that
$$
E\left[ d_1(\hat F_{0,m}, F_0) \right] \leq K_\alpha m^{-1+1/\alpha}.
$$

Now suppose that \eqref{eq:InductionHyp} holds for $j-1$. Let $\{U_{r}^i \}_{i,r \geq 1}$ be a sequence of i.i.d. Uniform$(0,1)$ random variables, independent of everything else. Let $\{(Q_i, N_i, C_{(i,1)}, C_{(i,2)}, \dots)\}_
{i \geq 1}$ be a sequence of i.i.d. copies of the generic branching vector, also independent of everything else. Recall that $F_{j-1}$ is the distribution function of $R^{(j-1)}$ and define the random variables
    \begin{align*}
        \hat R_{i}^{(j,m)} = \sum_{r=1}^{N_i} C_{(i,r)} \hat F_{j-1,m}^{-1}(U_r^i)+Q_i \qquad \text{and} \qquad R_{i}^{(j)} = \sum_{r=1}^{N_i} C_{(i,r)} F_{j-1}^{-1}(U_r^i)+Q_i
    \end{align*}
for each $i  = 1, 2, \dots, m$. Now use these random variables to define
$$\hat F_{j,m}(x) = \frac{1}{m} \sum_{i=1}^m 1(\hat R_i^{(j,m)} \leq x) \qquad \text{and} \qquad F_{j,m}(x) = \frac{1}{m} \sum_{i=1}^m 1(R_i^{(j)} \leq x).$$
Note that $F_{j,m}$ is an empirical distribution function of i.i.d. copies of $R^{(j)}$, which has been carefully coupled with the function $\hat F_{j,m}$ produced by the algorithm.

 By the triangle inequality and Proposition \ref{P.empirical_d1} we have that
  \begin{align*}
    E\left[ d_1(\hat F_{j,m}, F_{j} )\right] &\leq E\left[ d_1(\hat F_{j,m},  F_{j,m})\right] + E\left[ d_1(F_{j,m}, F_{j}) \right]  \leq  E\left[ d_1(\hat F_{j,m},  F_{j,m})\right] + K_\alpha m^{-1+1/\alpha}.
 \end{align*}
  To analyze the remaining expectation note that
 \begin{align*}
E\left[ d_1(\hat F_{j,m}, F_{j,m}) \right] &=  E\left[ \int_{-\infty}^\infty | \hat F_{j,m} (x) -  F_{j,m}(x) | dx \right] \\
&\leq \frac{1}{m} \sum_{i=1}^m E\left[ \int_{-\infty}^\infty  \left| 1(\hat R_i^{(j,m)} \leq x ) - 1( R_i^{(j)} \leq x) \right| dx \right] \\
&=  \frac{1}{m} \sum_{i=1}^m E\left[\left| \hat R_i^{(j,m)} - R_i^{(j)} \right|  \right]\\
&= \frac{1}{m} \sum_{i=1}^mE\left[\left| \sum_{r=1}^{N_i} C_{(i,r)} ( \hat F_{j-1,m}^{-1}(U_r^i) - F_{j-1}^{-1}(U_r^i) ) \right|  \right]\\
&\le  E\left[ \sum_{r=1}^N |C_r| \right] E\left[ d_1(\hat F_{j-1,m},F_{j-1}) \right],
\end{align*}
where in the last step we used the fact that $ (N_i, C_{(i,1)}, C_{(i,2)}, \dots)$ is independent of $\left\{ U_r^i \right\}_{r \geq 1}$ and of $\hat F_{j-1,m}$, combined with the explicit representation of the Kantorovich-Rubinstein distance given in \eqref{eq:Explicit}.   The induction hypothesis now gives
\begin{align*}
E\left[ d_1(\hat F_{j,m}, F_{j} )\right] &\leq \rho_1 E\left[ d_1(\hat F_{j-1,m},F_{j-1}) \right] + K_\alpha m^{-1+1/\alpha} \\
&\leq K_\alpha m^{-1+1/\alpha} \rho_1 \sum_{i=0}^{j-1} \rho_1^i  + K_\alpha m^{-1+1/\alpha} \\
&= K_\alpha m^{-1+1/\alpha} \sum_{i=0}^{j} \rho_1^i.
\end{align*}
This completes the proof.
\end{proof}

Note that the proof of Theorem~\ref{T.d1-l1-convergence} implies that  $\hat R_i^{(j,m)} \to R_i^{(j)} = \sum_{r=1}^{N_i} C_{(i,r)} F_{j-1}^{-1}(U_r^i) + Q_i \stackrel{\mathcal{D}}{=} R^{(j)}$ in $L^1$-norm for all fixed $j \in \mathbb{N}$, and hence in distribution. In other words,
\begin{equation} \label{eq:WeakConvergence}
    P\left( \hat R^{(k,m)}_i \le x \right) \to F_k(x) \qquad \text{as } m \to \infty,
\end{equation}
for all $i= 1, 2, \dots, m$, and for any continuity point of $F_k$. This also implies that
\begin{equation} \label{eq:MeansConverge}
E\left[ \hat F_{k,m}(x) \right] = P\left( \hat R^{(k,m)}_1 \le x \right) \to F_k(x) \qquad \text{as } m \to \infty,
\end{equation}
for all continuity points of $F_k$.

Since our algorithm produces a pool $\mathcal P^{(k,m)}$ of $m$ random variables approximately distributed according to $F_k$, it makes sense to use it for estimating expectations related to $R^{(k)}$. In particular, we are interested in estimators of the form in \eqref{eq:hEstimator}. The problem with this kind of estimators is that the random variables in $\mathcal P^{(k,m)}$ are only conditionally independent given $\hat F_{k-1,m}$.

\begin{definition}
We say that $\Theta_n$ is a consistent estimator for $\theta$ if $\Theta_n \stackrel{P}{\to} \theta$ as $n \to \infty$, where $\stackrel{P}{\to}$ denotes convergence in probability.
\end{definition}

Our second theorem shows the consistency of estimators of the form in \eqref{eq:hEstimator} for a broad class of functions.


\begin{theorem} \label{thm:estimator}
    Suppose that the conditions of Proposition \ref{P.MomentsR} are satisfied for some $\beta > 1$. Suppose $h: \mathbb{R} \to \mathbb{R}$  is continuous and $|h(x)| \le C(1+|x|)$ for all $x\in \mathbb R$ and some constant $C>0$. Then, the estimator
$$\frac{1}{m} \sum_{i=1}^m h(\hat R^{(k,m)}_i) = \int_{\mathbb{R}} h(x) d \hat F_{k,m}(x),$$
where $\mathcal{P}^{(k,m)} = \left( \hat R_1^{(k,m)}, \hat R_2^{(k,m)}, \dots, \hat R_m^{(k,m)}\right)$, is a consistent estimator for $E[h(R^{(k)})]$.
\end{theorem}

\begin{proof}
%
For any $M>0$, define $h_M(x)$ as
\[
    h_M(x)=h(-M)1(x\le -M)+h(x)1(-M<x\le M)+h(M)1(x>M),
\]
and note that $h_M$ is uniformly continuous.
We then have
\begin{align}
    \left|\int_{\mathbb R} h(x)d\hat F_{k,m}(x)-\int_{\mathbb R} h(x)dF_{k}(x)\right|&\le
    2C\int_{|x| > M}(1+|x|)dF_k(x)+2C\int_{|x| > M} (1+|x|)d\hat F_{k,m}(x)\nonumber\\
    &\quad +\left|\int_{\mathbb R} h_M(x)d\hat F_{k,m}(x)-\int_{\mathbb R}  h_M(x)dF_{k}(x)\right|.\label{eq:truncation}
\end{align}
Fix $\epsilon > 0$ and choose $M_\epsilon > 0$ such that $E\left[ (|R^{(k)}| + 1) 1(|R^{(k)}| > M_\epsilon) \right] \leq \epsilon/(4C)$ and such that $-M_\epsilon$ and $M_\epsilon$ are continuity points of $F_k$. Define  $(\hat R^{(k,m)}, R^{(k)}) = (\hat F_{k,m}^{-1}(U), F_k^{-1}(U))$, where $U$ is a uniform $[0,1]$ random variable independent of $\mathcal{P}^{(k,m)}$. Next, note that $g(x) = 1 + |x|$ is Lipschitz continuous with Lipschitz constant one and therefore
\begin{align*}
\int_{|x| > M_\epsilon} (1+|x|)d\hat F_{k,m}(x) &=  (1+M_\epsilon) \left( \hat F_{k,m}(-M_\epsilon) + 1 - \hat F_{k,m}(M_\epsilon) \right) \\
&\hspace{5mm} + \int_{x < -M_\epsilon} \hat F_{k,m}(x) \, dx +  \int_{x > M_\epsilon} (1- \hat F_{k,m}(x) )dx \\
&\leq (1+M_\epsilon) \left( \hat F_{k,m}(-M_\epsilon) + 1 - \hat F_{k,m}(M_\epsilon) \right) +  d_1(\hat F_{k,m}, F_k) \\
&\hspace{5mm} + \int_{x < -M_\epsilon} F_{k}(x) \, dx +  \int_{x > M_\epsilon} (1- F_{k}(x) )dx \\
&= (1+M_\epsilon) \left(\hat F_{k,m}(-M_\epsilon) - F_k(-M_\epsilon) + F_k(M_\epsilon) - \hat F_{k,m}(M_\epsilon) \right) + d_1(\hat F_{k,m}, F_k) \\
&\hspace{5mm} + E\left[ (|R^{(k)}| + 1) 1(|R^{(k)}| > M_\epsilon) \right].
\end{align*}
Finally, since $h_{M_\epsilon}$ is bounded and uniformly continuous, then $\omega(\delta) = \sup\{ | h_{M_\epsilon}(x) - h_{M_{\epsilon}}(y) | : |x-y| \leq \delta\}$ converges to zero as $\delta \to 0$. Hence, for any $\gamma > 0$,
\begin{align*}
\left| \int_{\mathbb{R}} h_{M_\epsilon}(x) d\hat F_{k,m}(x) - \int_{\mathbb{R}} h_{M_\epsilon}(x) dF_k(x) \right| &\leq E\left[ \left. \left| h_{M_\epsilon}(\hat R^{(k,m)}) - h_{M_\epsilon}(R^{(k)}) \right| \right| \hat F_{k,m} \right] \\
&\leq \omega( m^{-\gamma}) + K_\epsilon E\left[ \left. 1\left( |\hat R^{(k,m)} - R^{(k)} | > m^{-\gamma} \right) \right| \hat F_{k,m} \right] \\
&\leq \omega(m^{-\gamma}) + K_\epsilon m^\gamma d_1(\hat F_{k,m}, F_k),
\end{align*}
where $2K_\epsilon = \sup\{ |h_{M_\epsilon}(x)|: x \in \mathbb{R}\}$. Choose $0 < \gamma < 1 - 1/\alpha$ for the $\alpha \in (1, 2)$ in Theorem  \ref{T.d1-l1-convergence} and combine the previous estimates to obtain
\begin{align*}
&E\left[ \left| \int_{\mathbb{R}} h(x) d \hat F_{k,m}(dx) - \int_{\mathbb{R}} h(x) dF_k(dx) \right| \right] \\
&\leq 2C(1+M_\epsilon) \left( E[\hat F_{k,m}(-M_\epsilon)] - F_k(-M_\epsilon) + F_k(M_\epsilon) - E[\hat F_{k,m}(M_\epsilon)] \right) \\
&\quad + \epsilon  + \omega(m^{-\gamma}) + (2 C + K_\epsilon m^\gamma) E\left[ d_1(\hat F_{k,m}, F_k) \right].
\end{align*}
Since
$$E[\hat F_{k,m}(-M_\epsilon)]  \to F_k(-M_{\epsilon}) \qquad \text{and} \qquad E[\hat F_{k,m}(M_\epsilon)]  \to F_k(M_{\epsilon})$$
by \eqref{eq:MeansConverge}, and $m^\gamma E\left[ d_1(\hat F_{k,m}, F_k) \right] \to 0$ by Theorem~ \ref{T.d1-l1-convergence}, it follows that
$$\limsup_{m \to \infty} E\left[ \left| \int_{\mathbb{R}} h(x) d \hat F_{k,m}(dx) - \int_{\mathbb{R}} h(x) dF_k(dx) \right| \right] \leq \epsilon.$$
Since $\epsilon > 0$ was arbitrary, the convergence in $L^1$, and therefore in probability, follows.
\end{proof}

\section{Numerical examples}\label{sec:numerical}

This last section of the paper gives a numerical example to illustrate the performance of our algorithm. Consider a generic branching vector $(Q, N, C_1, C_2, \dots)$ where the $\{C_i\}_{i \geq 1}$ are i.i.d. and independent of $N$ and $Q$, with $N$ also independent of $Q$.

Figure \ref{fig:numerical1} plots the empirical cumulative distribution function of 1000 samples of $R^{(10}$, i.e., $F_{10,1000}$ in our notation, versus the functions $\hat F_{10,200}$ and $\hat F_{10,1000}$ produced by our algorithm, for the case where the $C_i$ are uniformly distributed in $[0, 0.2]$, $Q$ uniformly distributed in $[0,1]$ and $N$ is a Poisson random variable with mean 3. Note that we cannot compare our results with the true distribution $F_{10}$ since it is not available in closed form. Computing $F_{10,1000}$ required 883.3 seconds using Python with an Intel i7-4700MQ $2.40$ GHz processor and $8$ GB of memory, while computing $\hat F_{10,1000}$ required only 2.1 seconds.  We point out that in applications to information ranking algorithms $E[N]$ can be in the thirties range, which would make the difference in computation time even more impressive.

\begin{figure}[t]
    \centering
    \includegraphics[scale=1.2]{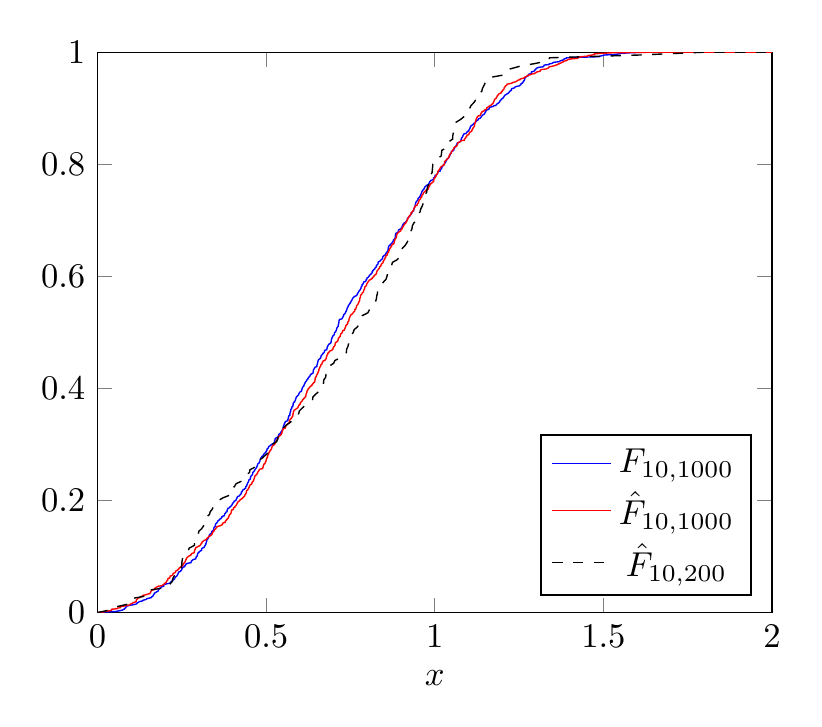}
    \caption{The functions $F_{10,1000}(x)$, $\hat F_{10,200}(x)$ and $\hat F_{10,1000}(x)$.}
    \label{fig:numerical1}
\end{figure}

Our second example plots the tail distribution of the empirical cumulative distribution function of $R^{(10)}$ for 10,000 samples versus the tail of $\hat F_{10,10000}$ for an example where $N$ is a zeta random varialbe with a probability mass function $P(N=k)\propto k^{-2.5}$, $Q$ is an exponential random variable with mean 1, and the $C_i$ have a uniform distribution in $[0,0.5]$. In this case the exact asymptotics for $P(R^{(k)} > x)$ as $x \to \infty$ are given by
$$P(R^{(k)} > x) \sim  \frac{(E[C_1] E[Q])^\alpha}{(1-\rho_1)^\alpha} \sum_{j=0}^{k} \rho_\alpha^j (1-\rho_1^{k-j})^\alpha P(N > x),$$
where $P(N > x) = x^{-\alpha} L(x)$ is regularly varying (see Lemma 5.1 in \cite{Jel_Olv_10}), which reduces for the specific distributions we have chosen to
$$G_{10}(x) \triangleq \frac{(0.25)^{2.5}}{(1 - (0.49))^{2.5}} \sum_{j=0}^{10}(0.07)^{j}(1-(0.49)^{10-j})^{2.5}      P(N > x) = (0.365) \, P(N>x).$$
Figure \ref{fig:numerical2} plots the complementary distributions of $F_{10,10000}$, $\hat F_{10,10000}$ and compares them to $G$. We can see that the tails of both $F_{10,10000}$ and $\hat F_{10,10000}$ approach the asymptotic roughly at the same time.

\begin{figure}[t]
    \centering
    \includegraphics[scale=1.2]{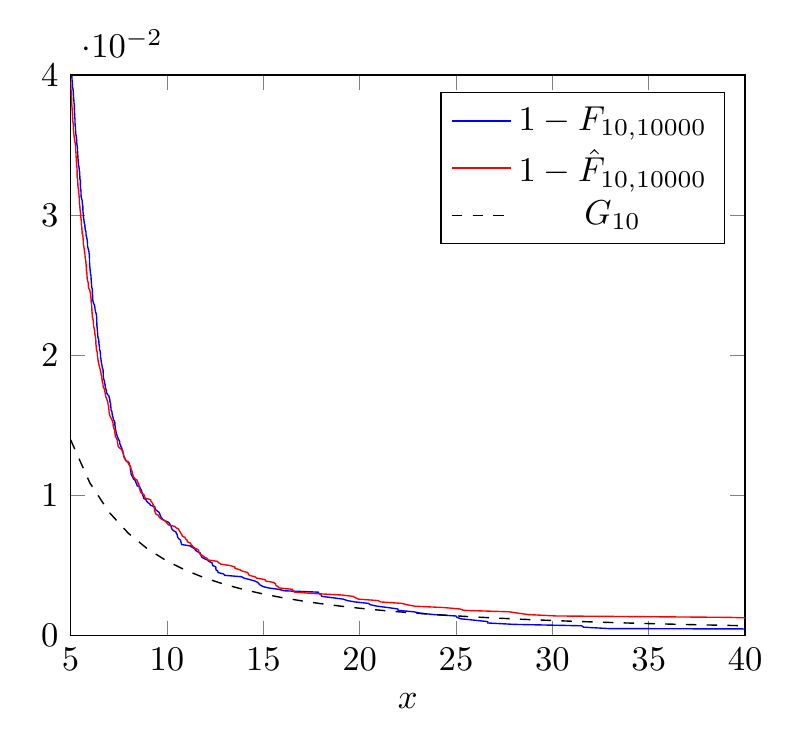}
    \caption{The functions $1 - F_{10,10000}(x)$, $1- \hat F_{10,10000}(x)$ and $G_{10}(x)$, where $G_{10}$ is evaluated only at integer values of $x$ and linearly interpolated in between.}
    \label{fig:numerical2}
\end{figure}

\bibliographystyle{plain}
\bibliography{bootstrapref}

\end{document}